\documentclass[12pt,leqno]{amsart}
\usepackage{fullpage}
\usepackage{amsmath}
\usepackage{amsfonts}
\usepackage{amssymb}
\usepackage{amsthm}
\usepackage{mathrsfs}
\usepackage{graphicx,xy}
\usepackage{enumerate}
\DeclareMathAlphabet{\mathpzc}{OT1}{pzc}{m}{it}

\def\bb{\mathbb}
\def\cal{\mathcal}

\def\F2{\mathbb F_2}
\def\A2{{\mathcal A}_2}

\def\C{{\mathbb C}}

\def\SC{{\cal{SC}}}

\def\cl{\mathpzc c}
\def\op{\mathpzc o}

\newtheorem{thm}{Theorem}[subsection]
\newtheorem{lem}[thm]{Lemma}
\newtheorem{prop}[thm]{Proposition}
\newtheorem{cor}[thm]{Corollary}
\theoremstyle{definition}
\newtheorem{defn}[thm]{Definition}

\theoremstyle{remark}
\newtheorem{rem}[thm]{Remark}

\theoremstyle{remark}

\xyoption{all}
\begin{document}

\title[Kontsevich Compactification and the Swiss-cheese Operad]{Explicit Homotopy Equivalences Between Some Operads}
\author{Eduardo Hoefel}
\address{Universidade Federal do Paran\'a, Departamento de Matem\'atica C.P. 019081, 81531-990 Curitiba, PR - Brazil }
\email{hoefel@ufpr.br}
\subjclass[2000]{}
\date{\today}
\thanks{The author is grateful to Muriel Livernet for valuable conversations leading him to fix lots of inaccuracies. This work is partially supported by 
the MathAmSud project OPECSHA-10MATH01 and Funda\c c\~ao Arauc\'aria, project FA-490/16032}
\begin{abstract}
In this work we present an explicit operad morphism that is also a homotopy equivalence between the operad given by the Fulton MacPherson compactification of configuration spaces and the little $n$-disks operad. In particular, the construction gives an operadic homotopy equivalence between the associahedra and the little intervals explicitly. It can also be extended to the case of Kontsevich compactification and Voronov swiss-cheese operad.
\end{abstract}
\maketitle

%%
%%
%%--------------------------------------------INTRODUCTION----------------------------------------
%%
%%

The little cubes operad $\cal C_n$ were introduced by Boardman and Vogt \cite{BoaVog73} and extensively used by many authors, including 
Peter May in his famous proof of the Recognition Principle of $n$-fold loop spaces \cite{May72}. On the other hand, the real version of the Fulton MacPherson 
compactification of configuration spaces of points were defined by Axelrod and Singer (see: \cite{AxelSing94} where the manifold with corners structure 
is presented in detail). In the case of the euclidean $n$-space, the Axelrod-Singer compactification results in a operad $\cal F_n$. This operad has also been studied by many authors. Here we will just mention Markl's characterization of $\cal F_n$ as an operadic completion \cite{Markl99c} and Salvatore's proof of 
its cofibrancy \cite{Salvatore01}. It is also well known that $\cal F_1$ gives Stasheff's Associahedra \cite{Stasheff63}. As a consequence of the cofibrancy proven by Salvatore, the operads $\cal F_n$ and $\cal C_n$ are related by the existence of an operad morphism $\nu : \cal F_n \to \cal C_n$ that is also a homotopy equivalence, i.e. an operadic homotopy equivalence. For more details and an extensive historical review, we refer the reader to \cite{MSS02}.   

In this work we construct an operadic homotopy equivalence between $\cal F_n$ and $\cal D_n$ explicitly by using elementary techniques, where $\cal D_n$ is the little disk analogue of the little cubes operad. The constructions also applies to the Swiss-cheese operad and the Kontesevich compactification. The Swiss-cheese operad was originally defined by Voronov in \cite{Voronov99} and a slightly different definition was given by Kontsevich in \cite{kont:defquant-pub}. The difference between the two versions 
of the Swiss-cheese operad from the point of view of algebras over Koszul operads is explored in detail in \cite{LivHoe11pr}, where the second 
version is called the unital swiss-cheese operad. In this paper we will restrict attention to the unital Swiss-cheese operad and show that it is operadically homotopy equivalent to the Kontsevich compactification.  

In Section \ref{disk_swiss_cheese} we review the Little Disks and the (unital) Swiss-cheese operad. The Manifold with Corners structure, given by Axelrod and Singer, on the real version of the Fulton MacPherson Compactification of the configurations spaces is reviewed in Section \ref{coordinates} for the case of points in the Euclidean space. The Kontsevich Compactification is also defined in Section \ref{coordinates}. The construction of the explicit operadic homotopy equivalence is given in Section \ref{main}. 
%, it uses crucially the coordinate system given by the manifold with corners structure of the compactification.
It is assumed that the reader has some familiarity with the Fulton MacPherson or the Axelrod Singer compactification.

%%
%%%%%%%%%%%%%%%%%%%%%%%%%%%%%%-------------------Little disks and Swiss-Cheese-------------------%%%%%%%%%%%%%%%%%%%%%%%%%%%%%%
%%

\section{Little disks and Swiss-cheese} \label{disk_swiss_cheese}

%%
%%%%%%%%%%%%%%%%%%%%%%%%%%%%%%-------------------Little disks-------------------%%%%%%%%%%%%%%%%%%%%%%%%%%%%%%
%%

\subsection{Little disks}

Let $D^2$ denote the standard unit disk in the complex plane $\bb C$. By a configuration of $n$ disks in $D^2$ we mean a map 
\begin{equation*}
 d : \coprod_{1 \leqslant s \leqslant n} D^2_s \to D^2
\end{equation*}
from the disjoint union of $n$ numbered standard disks $D^2_1, \dots, D^2_n$ to $D^2$
such that $d$, when restricted to each disk, is a composition of translations and dilations. 
The image of each such restriction is called a little disk.
The space of all configurations of $n$ disks is denoted $\cal D_2(n)$ and is topologized 
as a subspace of $(\bb R^2 \times \bb R^+)^n$ containing the coordinates of the center and 
radius of each little disk. The symmetric group acts on $\cal D_2(n)$ by renumbering the disks. 
For $n=0$, we define $\cal D_2(0) = \emptyset$. 
The $\Sigma$-module $\cal D_2 = \{ \cal D_2(n) \}_{n \geqslant 0}$ admits a well known structure 
of operad given by gluing configurations of disks into little disks, see \cite{MSS02}.

%%
%%%%%%%%%%%%%%%%%%%%%%%%%%%%%%%%------------------Swiss cheese------------------%%%%%%%%%%%%%%%%%%%%%%%%%%%%%%%%
%%

\subsection{Swiss-cheese}

%\begin{defn} 
For $m,n \geqslant 0$ such that $m + n >0 $, let us define $\cal{SC}(n,m;\op)$ as the space of those 
configurations $d \in \cal D_2(2n + m)$ such that its image in $D^2$ is invariant under complex conjugation 
and exactly $m$ little disks are left fixed by conjugation.
A little disk that is fixed by conjugation must be centered at the real line, in this case
it is called {\it open}. Otherwise, it is called {\it closed}.  
The little disks in $\cal{SC}(n,m;\op)$ are labeled according the following rules:
\begin{enumerate}[i)]
\item Open disks have labels in $\{1, \dots, m \}$ and closed disks 
have labels in $\{ 1, \dots, 2n \}$. 
\item Closed disks in the upper half plane have labels in $\{ 1, \dots, n \}$. 
If conjugation interchanges the images of two closed disks, their labels must be congruent modulo $n$.
\end{enumerate}
%\end{defn}

There is an action of $S_n \times S_m$ on $\cal{SC}(n,m;\op)$ extending the action of
$S_n \times \{ e \}$ on pairs of closed disks having modulo $n$ congruent labels and the action of
$\{ e \} \times S_m$ on open disks. Figure \ref{swiss_disc} illustrates a point in the space 
$\cal{SC}(n,m;\op)$.

\begin{figure}
  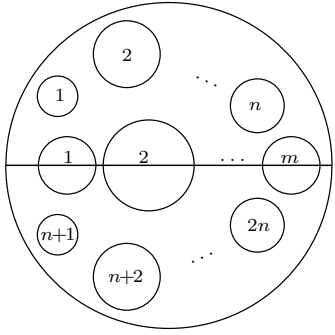
\caption{A configuration in $\cal{SC}(n,m;\op)$}
\label{swiss_disc}
\end{figure}

\begin{defn}[Swiss cheese operad]
 The 2-colored operad $\cal{SC}$ is defined as follows. For $m,n \geqslant 0$ with $m+n > 0$, $\cal{SC}(n,m;\op)$
is the configuration space defined above and $\cal{SC}(0,0;\op) = \emptyset$. For $n \geqslant 0$,  $\cal{SC}(n,0;\cl)$
is defined as $\cal D_2(n)$ and $\cal{SC}(n,m;\cl) = \emptyset$ for $m \geqslant 1$. 
The operad structure in $\cal{SC}$ is given by: 
\begin{align*}
& \circ_i^\cl: \SC (n,m;x)\times \cal{SC}(n',0;\cl)\rightarrow \cal{SC}(n + n'-1,0;x), \ \textrm{for}\ 1\leqslant i\leqslant n \\
& \circ_i^\op: \SC (n,m;x)\times \cal{SC}(n',m';\op)\rightarrow \cal{SC}(n + n',m + m' -1;x), \ \textrm{for}\ 1\leqslant i\leqslant m 
\end{align*}
When $x=\cl$ and $m = 0$, $\circ_i^\cl$ is the usual gluing of little disks in $\cal D_2$. 
If $x=\op$, $\circ_i^\cl$ is defined by gluing each configuration of $\cal{SC}(n',0;\cl)$ in the little disk labeled by $i$ 
and then taking the complex conjugate of the same configuration and gluing the resulting configuration in the little disk 
labeled by $i+n$. Since $\SC (n,m;\cl) = \emptyset$ for $m \geqslant 1$, 
$\circ_i^\op$ is only defined for $x = \op$ and is given by the usual gluing operation of $\cal D_2$. 
\end{defn}

%%%%%%%%%%%%%%%%%%%%%%%%%%%------Configurations of points in the upper closed half-plane------%%%%%%%%%%%%%%%%%%%%%%%%%%%%%%%%%%%

\section{Compactified Configurations Spaces}  \label{coordinates}

%Here we describe the homotopy equivalence between the Swiss cheese operad $\SC$ and the 
%Fulton-MacPherson compactification of the configuration space of points in the upper closed 
%half-plane introduced by Kontsevich in \cite{kont:defquant-pub}. We assume the reader has familiarity
%with the Fulton-MacPherson compactification of points in the complex plane
%(see \cite{AxelSing94,GetJon94,Merkulov10}). 
%In the case of points in the upper closed half-plane, a compactification following the guidelines of 
%Fulton-MacPherson was introduced by Kontsevich and will be referred to as the 
%Kontsevich's compactification in accordance with the terminology used in \cite{Merkulov10}). 

Let $p,q$ be non-negative integers satisfying the 
inequality $2p+q \geqslant 2$. We denote by Conf$(p,q)$ the configuration space of marked points
on the upper closed half-plane $H = \{ z \in \mathbb{C} \;|\; {\rm Im}(z) \geqslant 0 \}$ with 
$p$ points in the interior and $q$ points on the boundary (real line): 
\begin{equation*} 
    {\rm Conf}(p,q) = \{ (z_1, \dots, z_p, x_1, \dots, x_q) \in H^{p+q} \;|\;
   z_{i_1} \neq z_{i_2},\, x_{j_1} \neq x_{j_2} \ {\rm Im}(z_i) > 0,\,  {\rm Im}(x_j) = 0  \}  
\end{equation*}

The above configuration space ${\rm Conf}(p,q)$ is the Cartesian product of an open subset of $H^p$ and 
an open subset of $\mathbb{R}^q$ and, consequently, is a $2p + q$ 
dimensional smooth manifold. 
Let $C(p,q)$ be the quotient of ${\rm Conf}(p,q)$ by the action of the 
group of orientation preserving affine transformations that leaves the real line fixed: 
$C(p,q) = {\rm Conf}(p,q)\Big/(z \mapsto az + b)$ where 
$a,b \in \mathbb{R}, \; a > 0$. 
The condition $2p+q \geqslant 2$ ensures that the action 
is free and thus $C(p,q)$ is a $2p + q - 2$ dimensional smooth manifold. In the case of points in the 
complex plane we have: ${\rm Conf}(n) = \{ (z_1, \dots, z_n) \in \C^n \;|\; z_i \neq z_j, \forall i \neq j \}$
and $C(n) = {\rm Conf}(n)\Big/(z \mapsto az + b)$ where 
$a \in \mathbb{R}, \; a > 0 \makebox{ and } b \in \C$. The manifold $C(n)$ is $2n - 3$ dimensional and its 
real Fulton-MacPherson compactification is denoted by $\overline{C(n)}$ (see: \cite{AxelSing94}). 

Let $\phi$ be the embedding
$\phi : C(p,q)  \longrightarrow C(2p + q)$
defined by 
\begin{equation}\label{embedd}
\phi(z_1 , \dots, z_p, x_1, \dots, x_q) = 
(z_1 ,\bar{z}_1, \dots, z_p, \bar{z}_p, x_1, \dots, x_q) 
\end{equation}
where $\bar z$ denotes complex conjugation. 
%\begin{definition}
The Fulton-MacPherson compactification of $C(p,q)$ is defined as the 
closure in $\overline{C(2p + q)}$ of the image of $\phi$ and 
is denoted by $\overline{C(p,q)}$.
For a detailed combinatorial and geometrical study of $\overline{C(p,q)}$, we refer the reader to \cite{Devadoss11}.

%\end{definition}

Both compactifications $\overline{C(n)}$ and $\overline{C(p,q)}$ have the structure of manifolds 
with corners whose boundary strata are labeled by trees (for details, see: \cite{Hoefel09,Sinha04,kont:defquant-pub,Merkulov10}).
Such labelling by trees defines a 2-colored operad structure denoted by $\cal H_2$. The set 
of colors is $\{ \op,\cl \}$ and 
\begin{equation}
\cal H_2(p,q;x) := \left\{ 
\begin{array}{ll}
 \overline{C(p,q)}, &\  \textrm{if}\  x = \op  \         \textrm{and}\   2p + q \geqslant 2, \\
 \overline{C(p)},   &\  \textrm{if}\  x = \cl, \ q = 0 \ \textrm{and}\   p \geqslant 2,      \\
 \emptyset,         &\  \textrm{if}\  x = \cl  \         \textrm{and}\   q \geqslant 1.
\end{array}\right.
\end{equation}
In addition, we define $\cal H_2(1,0;\cl)$ and $\cal H_2(0,1;\op)$ as the one point space. 

On the other hand, the sequence of manifolds $\{ \overline{C(n)} \}_{n \geqslant 1}$ gives the well known operad $\cal F_2$, where $\overline{C(1)}$ is defined as the one point space.
The manifold $\overline{C(2)}$ is the circle $S^1$ while $\overline{C(3)}$ is the
3-manifold shown in Figure \ref{Jacobi_manifold} which can be described as a solid torus with a braid of two solid torus removed
from its interior. Since the three boundary components of $\overline{C(3)}$ are equivalent to tori, the boundary define three ways of embedding 
$S^1 \times S^1$ into $\overline{C(3)}$ which is part of the operad structure of $\cal F_2$. The manifold $\overline{C(3)}$ which is called the Jacobi manifold because its fundamental class provides a parametrization for the Jacobiator $J$, the homotopy operator for the Jacobi identity in a $L_\infty$-algebra.

%%
%%%%%%%%%%%%%%%%%%%%%%---------Cordinates on $\overline{C(n)}$---------
%%

\subsection{Coordinates on $\overline{C(n)}$}

Before we proceed, let us review some properties of $\overline{C(n)}$. %(see \cite{AxelSing94, KSV96}). 
The codimension $k$ boundary component of $\overline{C(n)}$ will be denoted by $\partial_k \overline{C(n)}$, 
it consists of a disjoint union of open submanifolds. More explicitly, we have: 
\begin{equation}
 \partial_k \overline{C(n)} = \bigsqcup_{|T| = k} C(n)(T) 
\end{equation}
where the disjoint union is taken for all labelled trees $T$ and $|T|$ denotes the number of internal edges of $T$.
Each stratum $C(n)(T)$ is open in $\partial_k \overline{C(n)}$ and the strata satisfy the following properties: 
\begin{enumerate}[1)]
 \item If $T$ is a corolla $\delta_k$, then $C(n)(\delta_k)$ is homeomorphic to $C(k)$;
 \item If $T = S_1 \circ_i S_2$, then $C(n)(S_1 \circ_i S_2)$ is homeomorphic to $C(n)(S_1) \times C(n)(S_2)$.
\end{enumerate}

It is also worth mentioning that the closure of each stratum is given by: 
\[ \overline{C(n)(T)} = \bigsqcup_{T' \to T} C(n)(T') \]
where $T' \to T$ means that $T$ can be obtained from $T'$ by contracting a finite number of internal edges.
Hence, the closure of $\partial_k \overline{C(n)}$ is $\bigsqcup_{|T| \geqslant k} C(n)(T)$. 

After modding out by translations and dilations, a configuration $\vec z \in C(p)$ may be seen as 
a sequence of pairwise distinct points $(z_1, \dots, z_p) \in \mathbb C^{\times p}$ that is in normal 
form, i.e., such that: 
$\sum_{i \in [p]} z_i = 0$ and $\sum_{i \in [p]} |z_i|^2 = 1$. In order to show that $\overline{C(n)}$ is 
a manifold with corners, Axelrod and Singer define (\cite{AxelSing94}, formula 5.71), for each tree $S$ with $n$ leaves and $k$ internal edges, a map: 
\begin{equation}
\cal M_S : C(n)(S) \times (\mathbb R_{\geqslant 0})^k \to \mathbb C^n.  
\end{equation}
In \cite{AxelSing94}, the points are in a Riemannian manifold and the coordinate system is defined through the exponential.  
In our case, the manifold is $\mathbb{C}$ and the exponential map is hidden in the affine structure of the complex plane.
The family of maps  $\cal M_S$ are characterized by the following properties: 
\begin{enumerate}[\it i)]
 \item For a corolla $\delta_n$, it is defined as the identity
  $ \cal M_{\delta_n} = {\rm Id} : C(n) \to C(n) \subseteq \mathbb C^n; $
 \item if $\cal M_S$ and $\cal M_T$ are already defined, where $S$ is a tree with $n_1$ leaves and $k$ 
       internal edges and $T$ is a tree with $n_2$ leaves and $l$ internal edges, 
       then $\cal M_{S \circ_i T}$ is defined as follows. First identify 
       $C(n)(S \circ_i T) \times (\mathbb R_{\geqslant 0})^{(k+l-1)} = 
        C(n)(S) \times (\mathbb R_{\geqslant 0})^k \times 
        C(n)(T) \times (\mathbb R_{\geqslant 0})^l \times R_{\geqslant 0}$, and then define:
\[ 
 \xymatrix@1@C=1cm{
  C(n)(S \circ_i T) \times (\mathbb R_{\geqslant 0})^{(k+l-1)} 
  \ar[rr]^-{\cal M_S \times \cal M_T \times {\rm Id}_{\mathbb R}} & &
  C(n_1) \times C(n_2) \times R_{\geqslant 0}
  \ar[r]^-{\gamma_i} & \mathbb C^n}
\]
where $n = n_1 + n_2 - 1$ and $\gamma_i : C(n_1) \times C(n_2) \times R_{\geqslant 0} \to \mathbb C^n$ is given 
by: 
\begin{equation}\label{eq:gamma-i}
\gamma_i (\vec x, \vec y, t) = (x_1, \dots, x_{i-1}, x_i + t(y_1,\dots,y_{n_2}), x_{i+1}, \dots, x_{n_1}), 
\end{equation}
where $\vec x = (x_1, \dots, x_{n_1}) \in C(n_1)$ and $\vec y = (y_1, \dots, y_{n_2}) \in C(n_2);$
 \item the maps $\cal M_S$ are $\Sigma_n$-equivariant in the following sense:
       \[ 
	  \cal M_{(S\sigma)} = (\cal M_S) \sigma, \quad \forall \sigma \in \Sigma_n, 
       \]
where the $\sigma$ action on the left hand side is the right $\Sigma_n$-action on trees, while the action on the right hand side 
the right $\Sigma_n$-action on $\mathbb C^n$. 
\end{enumerate}

\begin{rem}
The reader should compare the above $\gamma_i$ maps with Markl's pseudo-operad structure on Conf$(n)$ \cite{Markl99c}.
\end{rem}

%The next fact, proven by Axelrod and Singer, imply that the coordinates in $C(n)$ (given by the cartesian coordinates on $\mathbb R^2$)
%can be used to induce a coordinate system on $\overline{C(n)}$. 

\begin{figure}
 \includegraphics[scale=0.3,angle=90]{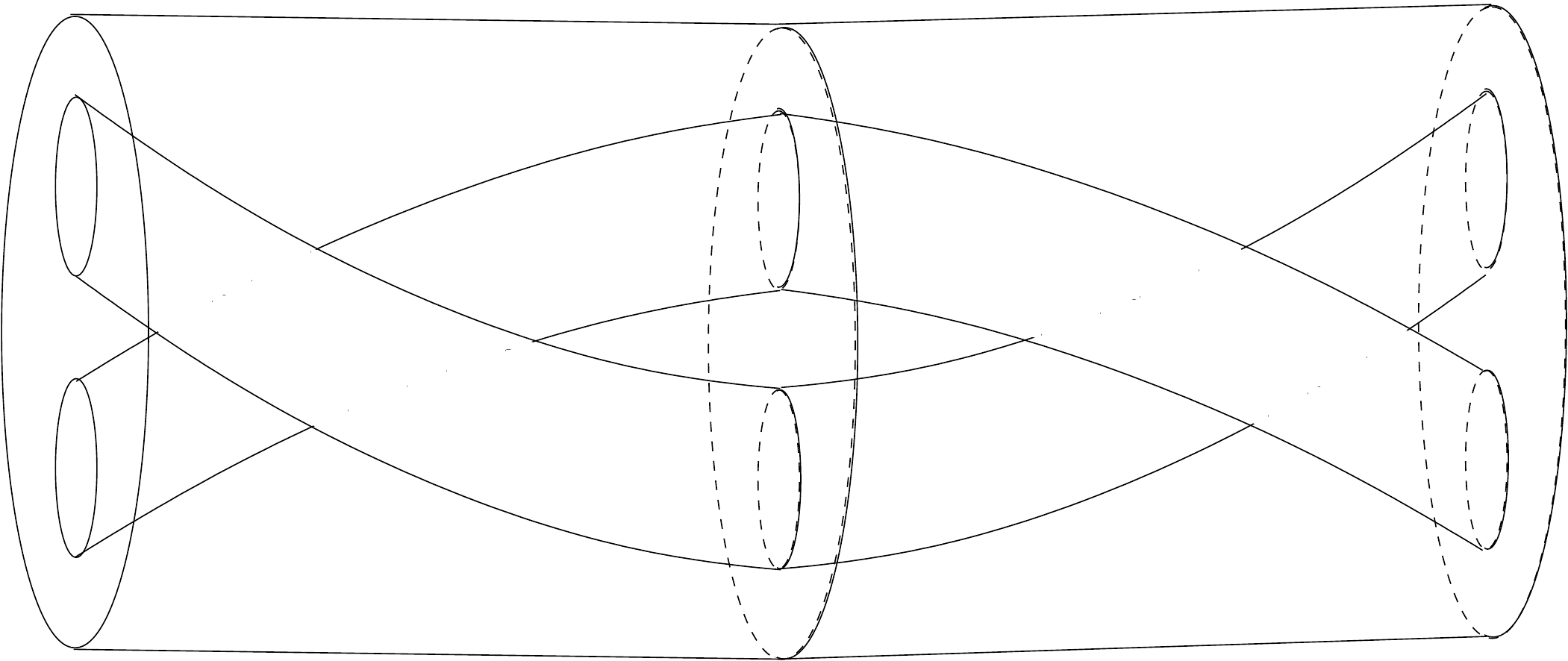}
 \caption{The manifold $\overline{C(3)}$ is obtained from the 3-manifold shown in this picture after identifying the two sides of the cylinder through the identity map.} 
 \label{Jacobi_manifold}
\end{figure}

Modding out by translations and dilations if necessary, we can assume that the local $\cal M_S$ maps 
assume values in $C(n)$. Axelrod and Singer showed that the local $\cal M_S$ maps can be continuously 
extended to maps of the form $\cal M_S : U \times W \to \overline{C(n)}$ and that this set of local $\cal M_S$ maps 
define a coordinate system on $\overline{C(n)}$ giving it a structure of manifold with corners (see also: \cite{Merkulov10}).

\begin{prop}[Axelrod-Singer]
 For any $n$-tree $S$ with $k$ internal edges and any point $p \in C(n)(S)$, there is an open neighborhood $U$ of $p$ 
 in $C(n)(S)$ and an open neighborhood $W$ of $0$ in $(\mathbb{R}_{\geqslant 0})^k$ such that $\cal M_S$ maps 
 $U \times (W \setminus \partial W)$ into ${\rm Conf}(n)$ and is a diffeomorphism onto its image.   
\end{prop}

%%
%%%%%%%%%%%%%%%%%%%%%%%%------------------------------Homotopy equivalence------------------------------%%%%%%%%%%%%
%%

\section{Operadic Homotopy Equivalence}  \label{main}

The explicit homotopy equivalence will use the coordinate system defined by the local $\cal M_S$ maps. The basic idea is to define the map from 
$\overline{C(n)} \to \cal D_2(n)$ in the obvious way on the interior of $\overline{C(n)}$ and extend it to the boundary as an operad morphism. The continuity problem can be solved through a collar neighborhood around the boundary.

%%
%%%%%%%%%%%%%%%%%%%%%%%%------------------------------Collar Neighborhood------------------------------%%%%%%%%%%%%
%%

\subsection{Collar Neighborhood}
Let $U$ be a collar neighborhood of $\partial \overline{C(n)}$ in $\overline{C(n)}$, with an homeomorphism
\begin{equation}
h : \partial \overline{C(n)} \times [0,1) \to U \subseteq \overline{C(n)},
\end{equation}
such that for any $p \in \partial \overline{C(n)}$ there is a neighborhood $W$ of p, 
such that $h(W \times [0,1))$ is a coordinate neighborhood of $p$ in $\overline{C(n)}$.
For any such $p \in \partial \overline{C(n)}$, the subset $h(p \times [0,1))$ is called the fiber 
of $p$ in the collar $U$ and $h(p \times (0,1))$ is the open fiber of $p$ in the collar $U$. 
In view of the description of the coordinate system in $\overline{C(n)}$ given by the local $\cal M_S$ maps, 
all the configurations in a fiber are obtained from the infinitesimal components of $p$ by applying the composition $\gamma_i$ a finite number of times. 

Let us denote by $\pi : \cal D_2(n) \to C(n)$ the canonical projection taking each configuration
of little disks into the configuration of their centers modded out by translations and dilations.

\begin{lem}
 For any $\vec x \in C(n)$, the inverse image $\pi^{-1}(\vec x)$ is convex in $\cal D_2(n)$.
\end{lem}
\begin{proof}
 It is enough to show that if $d_1$ and $d_2$ are two configurations of little disks in $\cal D_2(n)$ such that the centers
 of $d_1$ and $d_2$ define two configurations of points in $C(n)$ that are the same modulo translation and dilation then
 \begin{equation}
  \delta d_1 + (1 - \delta) d_2
 \end{equation}
 gives a well defined configuration of little disks in $\cal D_2(n)$ for all $\delta \in [0,1]$. Indeed, note that 
 the configurations can be presented in terms of centers and radii as follows: 
 \[ 
  d_1 = ((a_1, \alpha_1), \dots, (a_n, \alpha_n)) \quad \mbox{ and } \quad d_2 = ((b_1, \beta_1), \dots, (b_n, \beta_n)).
 \]
The disjointness between the interiors of two disks is given by: 
\begin{equation}\label{eq:disk-condition}
 \| a_i - a_j \| \geqslant \alpha_i + \alpha_j \quad \mbox{ and } \quad \| b_i - b_j \| \geqslant \beta_i + \beta_j.
\end{equation}
We denote by $\vec a$ and $\vec b$ the configurations of the centers in $d_1$ and $d_2$. Since 
$\vec a = \lambda \vec b + d$ for some $\lambda > 0$ and $d \in \mathbb C$, a straightforward computation 
shows that:
\[ 
\| (\delta a_i + (1 - \delta)b_i) - (\delta a_j + (1 - \delta)b_j) \| \geqslant
(\delta \alpha_i + (1 - \delta)\beta_i) + (\delta \alpha_j + (1 - \delta)\beta_j).
\]
Hence $\delta d_1 + (1 - \delta) d_2$ is a well defined configuration of little disks in $\cal D_2(n)$.
\end{proof}

\begin{cor}\label{well_defined}
For all $p \in \partial \overline{C(n)}$ and $d_1, d_2 \in \pi^{-1}[h(p \times [0,1))]$ any convex combination
\[ \delta d_1 + (1 - \delta) d_2, \qquad \delta \in [0,1] \] 
gives a well defined configuration in $\cal D_2(n)$.
\end{cor}
\begin{proof}
In the previous lemma we have seen that if the centers of little discs are related by translations and dilations, 
then the convex combination of the two configurations of little disks is well defined in $\cal D_2$. 
From the definition of the local $\cal M_S$ maps in the previous section, 
if the centers of $d_1$ and $d_2$ are in the same fiber of the tubular neighborhood, it follows 
that one is obtained from the other by a sequence of translations and dilations.  
The result then follows from the previous lemma.
\end{proof}

\begin{thm}
It is possible to construct an operadic homotopy equivalence $\nu : \cal F_2 \to \cal D_2$.
\end{thm}
\begin{proof}
The open submanifold $\overline{C(n)} \setminus \partial \overline{C(n)}$ of $\overline{C(n)}$ is 
homotopy equivalent to $C(n)$  which in turn is homeomorphic to 
the configuration space of $n$ points in the plane modded out by translations and dilations.
%(also denoted by \;\makebox{\raisebox{8pt}{$\circ$} \hskip -14pt $\mathcal{F}_2(n)$}). 
After modding out by translations and dilations, the configurations $(x_i)_{i \in [n]}$ can be 
thought of as configurations in normal form, i.e., such that 
$\sum_{i \in [n]} x_i = 0$ and $\sum_{i \in [n]} |x_i|^2 = 1$. By assigning to each point $x_i$ 
a disk centered at it with radius $r = {\rm min}\{ | x_i - x_j|, 1 - | x_i | \}_{1 \leqslant i < j \leqslant n}$, 
we get a continuous map 
$\nu(n)_1 : C(n) \to \mathcal{D}(n)$ which is clearly a homotopy equivalence.

Note that $\overline{C(2)}$ is just the circle $S^1$, hence a manifold without boundary. So 
$\overline{C(2)} = C(2)$, and in this case the map is already defined. 
Now, assuming that those maps are already extended 
for all $\overline{C(k)}$ with $k < n$, let us show how to extend them to $\overline{C(n)}$.
Since the boundary of $\overline{C(n)}$ has only strata that are products of $\overline{C(k)}$ with $k < n$, 
we define $\nu(n)_2 : \partial \overline{C(n)} \to \cal D_2(n)$ as an operad morphism. 

Now take a collar neighborhood $U$ around the boundary in $\overline{C(n)}$ given by the coordinate system of the previous section
and extend $\nu(n)_2$ to the collar neighborhood so that it is constant along each fiber.
%, and extend continuously the map to that neighbourhood.
Since $\overline{C(n)}$ is compact, there is a continuous function $u : \overline{C(n)} \to [0,1]$ that is $1$ on $\partial \overline{C(n)}$ and 
vanishes outside the collar neighborhood. We define \[ \nu(n) = (1 - u)\nu(n)_1 + u \nu(n)_2. \] For each $p$ in the collar $U$, we have that 
$\nu(n)_1(p)$ and $\nu(n)_2(p)$ belong to $\pi^{-1}[h(p \times [0,1))]$, hence the map $\nu(n)$ is well defined by Corollary \ref{well_defined}.
Since each $\nu(n)$ was defined as an operad morphism on the boundary and as a homotopy equivalence on the interior, we have an operad morphism
$\nu : \cal F_2 \to \cal D_2$ that is a homotopy equivalence for each $n$, because $\overline{C(n)}$
is homotopy equivalent to its interior $C(n)$. 
\end{proof}

\begin{rem}
With the same argument, one can construct an explicit operadic homotopy equivalence $\nu : \cal F_k \to \cal D_k$. When $k=1$, it is well known that $\cal F_1$ is the operad known as Stasheff's Associahedra, which is operadically homotopy equivalent to $\cal D_1$.
\end{rem}

\begin{cor}
There is a morphism of 2-colored operads 
$\mu : \cal H_2 \to \SC$ such that $\mu(p,q;x)$ 
is a homotopy equivalence for any $p+q \geqslant 1$ and $x = \cl,\op$.
\end{cor}
\begin{proof}
The manifold $\overline{C(p,q)}$ is embedded in $\overline{C(2p+q)}$ in the same way
that $\SC(p,q;o)$ is embedded in $\mathcal{D}_2(2p+q)$. Hence the operadic homotopy equivalence 
$\overline{C(2p+q)} \to \mathcal{D}_2(2p+q)$ naturally restricts to a homotopy equivalence between 
$\overline{C(p,q)}$ and $SC(p,q;o)$. 
\end{proof}

\bibliographystyle{amsplain}
\bibliography{NCG_BA}
\bigskip

\end{document}